\documentclass[pdflatex,sn-mathphys-num]{sn-jnl}
\usepackage{booktabs}
\usepackage{amsthm}
\usepackage{amsmath}
\usepackage{amsfonts, amssymb}
\usepackage{makeidx}
\usepackage{graphicx}
\usepackage{bm}
\usepackage{algorithm}
\usepackage{algorithmic}
\usepackage{color}
\usepackage{float}
\usepackage{orcidlink}
\usepackage{multirow}
\usepackage{multicol}
\usepackage{array}

\theoremstyle{thmstyleone}%
\newtheorem{thm}{Theorem}

\newtheorem{lem}[thm]{Lemma}

\newtheorem{defn}{Definition}
  
\newcommand{\eq}{\begin{equation}}
\newcommand{\nq}{\end{equation}}
\newcommand\eqa{\begin{eqnarray}}
\newcommand\nqa{\end{eqnarray}}

\newcommand\bess{\begin{eqnarray*}}
\newcommand\eess{\end{eqnarray*}}
\newcommand{\nn}{\nonumber}

\newcommand{\ALG}{}

\begin{document}

\title[Article Title]{Maximum Shortest Path Interdiction Problem by Upgrading Nodes on Trees under Unit Cost }

\author[1]{Qiao Zhang \footnote{The work of Q. Zhang was supported by National Natural Science Foundation of China (1230012046).}\orcidlink{0009-0004-4723-2017}}\email{qiaozhang@cczu.edu.cn}
\author[2]{Xiao Li \orcidlink{0009-0003-4222-6577}}\email{xli2576@wisc.edu}
\author*[3]{Xiucui Guan \orcidlink{0000-0002-2653-1868}}\email{xcguan@163.com}
\author[4,5]{Panos M. Pardalos \footnote{Research by P.M. Pardalos was prepared within the framework of the Basic Research Program
at the National Research University Higher School of Economics (HSE).}\orcidlink{0000-0001-9623-8053
}}\email{pardalos@ise.ufl.edu}

\affil[1]{Aliyun School of Big Data, Changzhou University, Changzhou 213164, Jiangsu, China}
\affil[2]{Department of Mathematics, University of Wisconsin, Madison, WI 53706, USA}
\affil*[3]{School of Mathematics, Southeast University, Nanjing 210096, Jiangsu, China}
\affil[4]{\orgdiv{Center for Applied Optimization, University of Florida}, 32611, Florida, USA}
\affil[5]{\orgdiv{Laboratory LATNA, HSE University, Russia}}
\abstract{Network interdiction problems by deleting critical nodes have wide applications. However, node deletion is not always feasible in certain practical scenarios. We consider the maximum shortest path interdiction problem by upgrading nodes on trees under unit cost (MSPIT-UN$_u$). It aims to upgrade a subset of nodes to maximize the length of the shortest root-leaf distance such that the total upgrade cost under unit cost is upper bounded by a given value. We develop a dynamic programming algorithm with a time complexity of $O(n^3)$ to solve this problem. Furthermore, we consider the related minimum cost problem of (MSPIT-UN$_u$) and propose a $O(n^3\log n)$ binary search algorithm, where a dynamic programming algorithm is exceeded in each iteration to solve its corresponding problem (MSPIT-UN$_u$). Finally, we design numerical experiments to show the effectiveness of the algorithms.  
}

\maketitle

\section{Introduction}
Shortest path interdiction problems involving the strategic deletion of critical edges or nodes (denoted as \textbf{(SPIP-DE/N)}) have garnered significant attention in the research community over the past two decades. These problems have found extensive applications across diverse domains, including transportation networks \cite{r1}, military operations \cite{Kha}, and the disruption of terrorist networks \cite{r2, Kha}. The fundamental objective of the interdiction problem is to identify and remove $K$ critical edges/nodes from a network to make the performance of the shortest path from some points $u$ and $v$ as poor as possible.

The shortest path interdiction problems by deleting critical edges raise more attention.  Corley and Sha \cite{CorleyandSha} pioneered K-edge removal to maximize shortest path length, proven NP-hard by Ball et al. \cite{Ball}. Israeli et al. \cite{Israeli} proposed an enhanced Benders algorithm, outperforming classical methods. Khachiyan et al. \cite{Kha} established inapproximability for threshold variants and extended results to node interdiction. Chen \cite{chen} and Nardelli et al. \cite{Nardelli} achieved O($m + n\log n$) and O($m\alpha(m,1)$) complexities for K=1 on undirected graphs. Ayyildiz et al. \cite{Ayyildiz} modeled multi-sink planner-disruptor interactions, while Huang et al. \cite{Huang} introduced RL-based solutions for grids/random graphs. 

While relatively less attention have been given to the shortest path interdiction problems by deleting critical nodes in the literature. Lalou et al. \cite{Lalou} presented a comprehensive survey on critical node problems, where the objective is to maximize network dispersion by removing strategic nodes. Their work analyzed computational complexity and developed algorithms for critical node problems under various dispersion metrics. Magnouche and Martin \cite{MM} from Huawei Technologies France investigated the minimal node interdiction problem, where the goal is to remove the minimum number of nodes such that the shortest $s-t$ path length in the residual graph is at least $d$. They proved the NP-hardness of this problem, formulated an integer linear programming model with exponentially many constraints, and developed a branch-and-bound algorithm for its solution.

It is noteworthy that the majority of existing research on shortest path interdiction problems has predominantly focused on the complete disabling of critical edges or nodes. However, in many real-world applications, the complete removal or neutralization of nodes is often impractical or infeasible. A more realistic and operationally viable alternative is to partially degrade the operational capacity or functionality of certain critical nodes. A compelling example of this can be observed in the context of wildfire management, as demonstrated by the 2025 LA Fires case \cite{wildfire}. In such scenarios, while the primary objective is to contain and mitigate the spread of wildfires, completely eliminating all potential ignition nodes is neither feasible nor cost-effective. 
In real-world scenarios, forests typically encompass numerous high-risk zones that serve as potential fire spread points, including  high-risk areas with dry vegetation, lightning-prone ridges, or zones with significant human activity.
Safeguarding all of them to a level of zero vulnerability is often impossible due to resource constraints and environmental complexities.
Therefore, a more pragmatic strategy is to selectively reduce the vulnerability of certain critical nodes, thereby delaying the propagation of fire through the network.  This corresponds to reducing the risk level of certain nodes rather than completely fireproofing them (equivalent to setting their vulnerability to zero). Motivated by such practical constraints, we adopt the concept of upgrading nodes introduced in \cite{ZhangQ21} and adapt it to analyze shortest path interdiction problems on trees. This approach provides a more realistic and operationally feasible framework for addressing interdiction challenges under real-world limitations. By focusing on the strategic enhancement of node capabilities rather than their complete removal or neutralization, our methodology not only improves the applicability of interdiction strategies in practical scenarios but also paves the way for innovative research in network optimization and resource allocation. This perspective bridges the gap between theoretical models and real-world implementation, offering new insights into solving complex interdiction problems.

The maximum shortest path interdiction problem by upgrading nodes on trees (denoted by \textbf{(MSPIT-UN)}) can be defined as follows.
Let $T=(V, E, w)$ be an edge-weighed tree rooted at $v_1$, where $V=\{v_1, v_2, \cdots, v_n\}$ and  $E=\{e_1, e_2, \cdots, e_m\}$ are the sets of vertices and edges, respectively. Let $Y=\{t_1, t_2, \cdots, t_l\}$ be the set of leaves. Let $c(v)$ denote the cost associated with upgrading node $v\in V$. Let $A(v_i ) = \{e _j = (v_i, v_j )|e_ j \in E\}$  be the set of
edges adjacent to $v_i$. For each edge $e\in E$,  let $w(e)$  and $u(e)$ be the original length and the length after upgrading its forward node, respectively, where $w(e)\leq u(e)$. Denote by $\Delta{w(e)}:=u(e)-w(e)$ the deviation between $u(e)$ and $w(e)$.   Denote by $P_{k}:=P_{t_k}:=P_{v_1,t_k}$ the unique path from $v_1$ to $t_k$ on  $T$ and define $w(P_{k})$ as the shortest root-leaf distance from $v_1$ to $t_k$ under the length vector $w$.
   The problem \textbf{(MSPIT-UN)} aims to upgrade a subset $S \subseteq V$ of nodes to maximize the length of shortest root-leaf  distance
such that the total upgrade cost under some norm is upper bounded by a given value
$K$. Its mathematical model can be stated as follows.

\begin{eqnarray}
&\max\limits_{S\subseteq V}&\min_{t_k\in{Y}}\bar{w}(P_k)\nonumber\\
\textbf{(MSPIT-UN)}&s.t.& \sum_{v\in S}c(v)\leq K,\label{eq-MSPIT-UN}\\
&&\bar w(e)=\left\{\begin{array}{ll}
 \beta(e),   & v\in S,e\in A(v) \\
   w(e),  & otherwise
\end{array}\right..\nonumber
\end{eqnarray}
The relevant minimum cost problem \textbf{(MSPIT-UN)}, denoted by \textbf{(MCSPIT-UN)},
aims to upgrade a subset $S \subseteq V$ of nodes to minimize the total upgrade cost such that
the shortest root-leaf distance is upper bounded by a given value $D$. Its mathematical
model can be stated as follows.
\begin{eqnarray}
&\min\limits_{S\subseteq V}&\sum_{v\in S}c(v)\nonumber\\
\textbf{(MCSPIT-UN)}&s.t.& \min_{t_k\in{Y}}\bar{w}(P_k)\geq D,\label{eq-MCSPIT-UN}\\
&&\bar w(e)=\left\{\begin{array}{cc}
  u(e),   & v\in S,e\in A(v) \\
   w(e),  & otherwise
\end{array}\right..\nonumber
\end{eqnarray}
Zhang, Guan et al.\cite{ZhangQ20,ZhangQ21} researched on the  maximum shortest path interdiction problem by upgrading edges on trees (denoted by \textbf{(MSPIT-UE)}). 
Under  weighted $ l_1 $ norm,  they\cite{ZhangQ20} proposed a primal-dual algorithms  in $ O(n^2) $ time for the problems \textbf{(MSPIT$_1$-UE)} and \textbf{(MCSPIT$_1$-UE)}. Under  weighted Hamming distance, they\cite{ZhangQ21} demonstrated that the two problems are  $\mathcal{NP}$-hard. While for the unit sum-Hamming distance,  they developed two dynamic programming  algorithms  with time complexity $ O(n^4) $  and $ O(n^4\log n) $ for the two problems.   Subsequently, with some local optimization, Yi et al. \cite{YiL22} improved upon these algorithm, achieving two reduced time complexity of  $ O(n^3) $  and $ O(n^3\log n) $ for the two  problem under unit sum-Hamming distance.

In this paper, we focus on the problems  (\textbf{MSPIT-UN}) and \textbf{(MCSPIT-UN)} under the unit cost assumption, where $c(v)=1,$ for all $v\in V$. These problems are denoted as (\textbf{MSPIT-UN$_u$}) and \textbf{(MCSPIT-UN$_u$)},respectively. We construct the models of the problems, analyze their properties, design optimization algorithms with time complexity analysis.

The paper is organized as follows.
In section 2, we first introduce some necessary definitions and structures. Then we propose a dynamic programming algorithm to solve the problem (\textbf{MSPIT-UN$_u$})  with time complexity $O(n^3)$.  Based on this problem and a binary method, in section 3,  we solve the problem (\textbf{MCSPIT-UN$_u$})  in $O(n^3\log n)$ time. In section 4,  numerical experments are conducted to show the effeciency of the two algorithms. In section 5, we draw a conclusion and put forward future research.

\section{Solve the problem (MSPIT-UN$_u$)}

According to model (\ref{eq-MSPIT-UN}), the problem  \textbf{(MSPIT-UN)} under unit cost\textbf{(MSPIT-UN$_u$)} can be formulated as the following form.
\begin{eqnarray}
&\max\limits_{S\subseteq V}&\min_{t_k\in{Y}}\bar{w}(P_k)\nonumber\\
 \textbf{(MSPIT-UN$_u$)}&s.t.& |S|\leq K,\label{eq-MSPIT-UNu}\\
&&\bar w(e)=\left\{\begin{array}{cc}
  u(e),   & v\in S,e\in A(v) \\
  w(e),  & otherwise
\end{array}\right..\nonumber
\end{eqnarray}

From model (\ref{eq-MSPIT-UNu}), it is evident that  the problem \textbf{(MSPIT-UN$_u$)} is aiming at  upgrading at most $K$ edges on a tree to maximize the shortest root-leaf distance of the tree.

In this section, based on model (\ref{eq-MSPIT-UNu}), we introduce several important definitions and a special data structure of left-subtree. Subsequently, we propose a dynamic programming algorithm with a time complexity of $O(n^3)$. Finally, we provide an illustrative example to demonstrate the execution of the algorithm.
\subsection{Some important definitions}
In this subection,  we introduce several important definitions and a special data structure of left $q$-subtree.
\begin{defn}\cite{ZhangQ20}\label{defn-tab}
	For $e_j=(v_i, v_j)$, we call $v_i$  the \textbf{father} of $v_j$, denoted by $father(v_j)=v_i$. Define $Layer(v_1):=1$ and the \textbf{Layer} of any other nodes $v\in {V\setminus \{v_1\}}$ as
	$$Layer(v):=\left\{
	\begin{array}{lr}
	Layer(father(v)), & \text{if }\ deg(v)\leq{2}\\
	Layer(father(v))+1, & \text{if }\ deg(v)>2
	\end{array} \right..$$
\end{defn}

\begin{defn}\cite{ZhangQ20}\label{defn-LN-edge}
	For each edge $e_j=(v_i,v_j)\in E$ with $Layer(v_i)\leq Layer(v_j),$ we define $LN(e_j):=Layer(v_i)$ as the \textbf{layer number} of the edge $e_j$.
\end{defn}
As is shown in \textbf{Figure \ref{LNCCV}}, $Layer(v_1): = 1$, $degree(v_2) := 3 > 2$, $father(v_2): = v_1$, thus,
$Layer(v_2): = 1 + 1 = 2$; $degree(v_5): = 2$, $father(v_5):= v_1$, thus, $Layer(v_5): = 1$. For edge $e_2: =
(v_1, v_2),$ $ Layer(v_1) \leq Layer(v_2)$, so $LN(e_2): = Layer(v_1) = 1$. 

For convenience, denote by $V^*: = \{v \in V\backslash\{v_1\}|degree $ $(v) > 2\}$ the set of nodes
whose degrees are more than 2.

\begin{defn}\cite{ZhangQ20}
	For a  node $\bar v\in V^*\cup\{v_1\}$, we define a set $CD(\bar v)$ of \textbf{critical descendant}.
	Let $\bar v$ be on the path from $v_1$ to $v\in V^*\setminus\{v_1\}\cup Y$. If $v\in V^*\setminus\{v_1\}$ and $Layer(v)=Layer(\bar v)+1$, then $v\in{CD(\bar v)}$; if $v\in Y$ and $Layer(v)=Layer(\bar v)$, then $v\in{CD(\bar v)}$. Correspondingly, if $v\in CD(\bar v)$,  we call $\bar v$ the \textbf{critical ancestor} of $v$, denoted by $CA(v):=\bar v$.
\end{defn}
For instance, in \textbf{Figure \ref{LNCCV}}, $CD(v_1) := \{v_2, v_6, v_7\}$  and $CA(v_2) := v_1$.
\begin{figure}
    \centering
\includegraphics[width=0.5\linewidth]{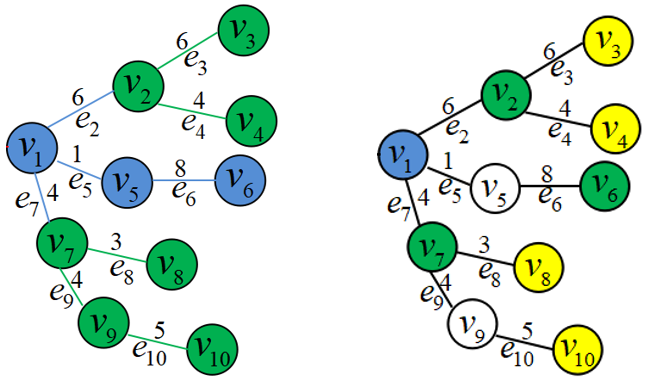}
    \caption{The edge-weighted trees $T_{v_1}$ with cost $c(e)$ on an edge $e$.  In the left tree, the Layers/layer numbers  of the blue nodes/edges are 1, and the Layers/layer numbers of the green nodes/edges are 2. In the right tree, the green nodes are the critical descendant of node $ v_1 $, and the paths are stored in green and yellow nodes.}
    \label{LNCCV}
\end{figure}

\begin{defn}\cite{ZhangQ20}\label{defn-chain}
	For any  node $v\in V^*\backslash\{v_1\}\cup Y$, define $\chi_v:=P_{CA(v),v}$ as the \textbf{chain} from $CA(v)$ to $v$. 
\end{defn}


For convenience, in the following parts of this paper, for any $v\in{V^*\cup\{v_1\}}$, let $CD(v)=\{v_{h_1}, v_{h_2}, \cdots, v_{h_{p}}\}$ be the set of critical children of the node $v$, where
\begin{eqnarray}p:=\left\{
\begin{array}{ll}
degree(v),&v=v_1\\
degree(v)-1,&v\in{V^*}\end{array}\right..\label{eq-p}\end{eqnarray}

\begin{lem}\cite{ZhangQ21}
Suppose $w'$ is an optimal solution of the problem ($\ref{eq-MSPIT-UNu}$). If there are two nodes $v_i,v_j$ on a same chain with $degree(v_i)=degree(v_j)=2$, $LN(A(v_i))=LN(A(v_j))$,
$\Delta{w(A(v_i))}<\Delta{w(A(v_j))}$, $w'(A(v_i))>w(A(v_i))$, $w'(A(v_j))=w(A(v_j))$, then $w^*$ is an optimal solution of the problem ($\ref{eq-MSPIT-UNu}$), where

$$ w^*(e)=\left\{
\begin{array}{ll}
w(e), &e=e_i\\
u(e), &e=e_j\\
w'(e),&otherwise
\end{array}\right..$$
\end{lem}







Based on the lemma above, without loss of generality, we can rearrange the edges with the same layer number on the same path such that their values of $\Delta w(e)$ are sorted non-increasingly. Notice that if the $K$ edges with the same layer number on the same path are upgraded, we can update the first $K$ edges in this layer of this path.


\begin{defn}\cite{ZhangQ21}
Define the left $q-$subtree rooted at $v$ as
$$
T_v^{1:q} =\bigcup\limits_{i=1}^{q}  T_v^{i:i}, \text{where~}   T_v^{i:i}:=\chi_{v_{h_i}} \cup T_v^{h_i}, i=1,2,\cdots,q, q=1,2,\dots,p 
$$ \text{and } $p$ \text{is defined as (\ref{eq-p}) }.  Specially, denote $ T_v^{1:p} $ as $ T_v $ and  let $ T_v^{1:0}:=\emptyset $. For any leaf node $v\in Y$, let $T_v^{1:q}:=\emptyset.$
\end{defn}

As illustrated in Fig. $\ref{lc1}$, the areas marked in red, blue, green correspond to the left 1-subtree $T_v^{1:1}$, the left 2-subtree $T_v^{1:2}$ and the left $q-$subtree $T^{1:q}_v$ of $T_v=T_v^{1:p}$, respectively.

\begin{figure}[!htbp]
\centering
\includegraphics[totalheight=1.2in]{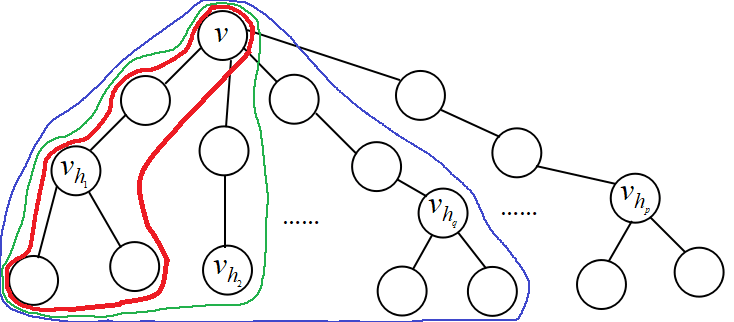}
\caption{The subtree $T^{1:p}_v$ is shown. The areas labeled in red,blue,green are the subtrees $T_v^{1:1},T_v^{1:2}$ and $T^{1:q}_v$, respectively.} \label{lc1}
\end{figure}
\subsection{Some auxiliary functions}
To describe the dynamic programming algorithm of the problem \textbf{(MSPIT-UN$_u$)}, we define the following functions.

\textbf{(1) The function $g(\chi_{v_{h}},\varepsilon,k)$ defined on the chain $\chi_{ v_h}$.}

For all $v\in{V^*}\cup\{v_1\}$ and every $v_h\in{CD(v)}$, let $\chi_{ v_h}:=\{e_{i_1},e_{i_2},$
$\cdots,e_{i_\beta}\}=\{v=v_{i_0},v_{i_1},v_{i_2},\cdots,v_h\}$ with $\Delta w(e_{i_1})\geq \Delta w(e_{i_2})\geq \cdots\geq\Delta w(e_{i_\beta})$, where $\beta=|V(\chi_{ v_h})|-1\geq 1$. Let  $g(\chi_{ v_h},\varepsilon,k)$ be  the sum of the edge lengths when the first $k$ nodes are upgraded on the chain $\chi_{ v_h}:=P_{v, v_{h}}$ with    $\varepsilon=0,1$, $k\geq \varepsilon$ and $0\leq k\leq \min\big\{\beta,K\big\},$  where $g(\chi_{ v_h},1,k)$ and  $g(\chi_{ v_h},0,k)$ represent the cases when  the node $v$ is and is not upgraded,respectively.  Let $V(\chi_{ v_h}, 0,k)$  and $V(\chi_{ v_h}, 1,k)$ be the corresponding sets of nodes which are upgraded on the chain $\chi_{ v_h}$, respectively.
 The values of $g(\chi_{ v_h},\varepsilon,k)$ can be calculated as follows.

\textbf{Case (1-1): }  when $k=0$, then $\varepsilon=0$.  
\begin{eqnarray}
  g(\chi_{ v_h},0,0):=\sum\limits_{j=1}^{\beta}w(e_{i_j}),    V(\chi_{ v_h},0,0)=\emptyset.\label{eq-g1v1}
\end{eqnarray}



\textbf{Case (1-2-1): }    when $K\geq k=\beta= 1$, then $\varepsilon=1$  and  $v$ is the only node can be upgraded. 
\begin{eqnarray}
  g(\chi_{ v_h},1,1)= w(e_{i_1})+ \Delta w(e_{i_1})=u(e_{i_1}),  \label{eq-g2}
\end{eqnarray}
\begin{eqnarray}
  V(\chi_{ v_h},1,1)=\{v\}.\label{eq-v2}
\end{eqnarray}

\textbf{Case (1-2-2): }    when $\beta\geq 2$, $1\leq k\leq \min\big\{\beta,K\big\},$ 
\begin{eqnarray}
  g(\chi_{ v_h},\varepsilon,k)=\left\{\begin{array}{cc}
   \sum\limits_{j=1}^{\beta}w(e_{i_j})+\sum\limits_{j=2}^{k+1} \Delta w(e_{i_j}),   &\varepsilon=0  \\
   \sum\limits_{j=1}^{\beta}w(e_{i_j})+\sum\limits_{j=1}^{k} \Delta w(e_{i_j}),   & \varepsilon=1
  \end{array}\right.,\label{eq-g3}
\end{eqnarray}
\begin{eqnarray}
  V(\chi_{ v_h},\varepsilon,k)=\left\{\begin{array}{cc}
  \bigcup\limits_{j=1}^{k}\{v_{i_j}\} ,   &\varepsilon=0  \\
  \bigcup\limits_{j=0}^{k-1}\{v_{i_j}\},   & \varepsilon=1
  \end{array}\right..\label{eq-v3}
\end{eqnarray}

Notice that we have considered all the chains on tree $T$.

\textbf{(2) The functions $F(T_v^{a:q},k)$ and $f(T_v^{a:q},\varepsilon,k)$ defined on the non-chain $T_v^{a:q}$.}

 We next discuss the two cases when $a=q$ and $a=1$. Specially, for any leaf node $v\in Y$, define $F(T_v,0):=F(\emptyset,0):=0$. 
 
\textbf{(2.1) When $a=q$, the functions $F(T_v^{q:q},k)$ and $f(T_v^{q:q},\varepsilon,k)$ defined on the non-chain $T_v^{q:q}$.}

For all $v\in{V^*}\cup\{v_1\}$, let  
 $F(T_v^{q:q},k)=\max\limits_{\varepsilon=0,1}f(T_v^{q:q},\varepsilon,k)$ be the maximum shortest root-leaf distance  of $T_v^{q:q}$ when $k$ nodes are upgraded with  $\varepsilon=0,1$, $k\geq \varepsilon$  ,  $0\leq{k}\leq\min\big\{{|V(T_v^{q:q})|-|Y\cap V(T_v^{q:q})|},K\big\} $ and $q=1,2,\cdots,p$, 
 where $f(T_v^{q:q},1,k)$ and $f(T_v^{q:q},0,k)$ represent the cases when  the node $v$ is and is not upgraded,respectively. Let  $V_v^{q:q}(k)$, $V_v^{q:q}(0,k)$ and $V_v^{q:q}(1,k)$  be the corresponding sets of edges which are upgraded on $T_v^{q:q}$, respectively.


\textbf{Case (1-1): } when  $k=0$,  then $\varepsilon=0$. 
\begin{eqnarray}
F(T_v^{q:q},0)=f(T_v^{q:q},0,0)=g(\chi_{v_{h_q}},0,0)+F(T_{v_{h_q}},0),\label{eq-f1}\end{eqnarray}
\begin{eqnarray}V_v^{q:q}(0)=V_v^{q:q}(0,0)=\emptyset.\label{eq-vf1}
\end{eqnarray}
Specially, when  $v_{h_q}$ is a leaf node, $T_{v_{h_q}}=\emptyset$ ,$F(T_{v_{h_q}},0)=0$ and then 
\begin{eqnarray}
&&F(T_v^{q:q},0)=f(T_v^{q:q},0,0)=g(\chi_{v_{h_q}},0,0),\label{eq-f2}\\
&&V_v^{q:q}(0)=V_v^{q:q}(0,0)=V(\chi_{v_{h_q}},0,0)=\emptyset.\label{eq-vf2}
\end{eqnarray}

\textbf{Case (1-2-1): }When $K\geq {k}={|V(T_v^{q:q})|-|Y\cap V(T_v^{q:q})|}$, then $\varepsilon=1$ as all nodes on $T_v^{q:q}$ are upgraded.
\begin{eqnarray}
&&F(T_v^{q:q},k)=f(T_v^{q:q},1,k)=\max\Big\{ g(\chi_{v_{h_q}},1,k_1)+F(T_{v_{h_q}},k_2)\Big\}\label{eq-f3},\\
&&s.t.~~~~k_1+k_2=k, \nn\\
&& k_1=1,2,\cdots, \min\big\{|V(\chi_{v_{h_q}})|-1,K\big\},\nn\\
&&k_2=0,1, 2, \cdots, \min\big\{|V(T_{v_{h_q}}|-|Y\cap V(T_{v_{h_q}})|,K\big\},\nn\\
&& V_v^{q:q}(k)=V_v^{q:q}(1,k).\label{eq-vf3}
\end{eqnarray}
Specially, when  $v_{h_q}$ is a leaf node, $T_{v_{h_q}}=\emptyset$ ,$F(T_{v_{h_q}},0)=0$ and then $k_2=0,$
\begin{eqnarray}
&&F(T_v^{q:q},k)=f(T_v^{q:q},1,k)=g(\chi_{v_{h_q}},1,k),\label{eq-f4}\\
&&V_v^{q:q}(k)=V_v^{q:q}(1,k)=V(\chi_{v_{h_q}},1,k).\label{eq-vf4}
\end{eqnarray}
\textbf{Case (1-2-2): }When 
$1\leq{k}\leq K<|V(T_v^{q:q})|-|Y\cap V(T_v^{q:q})|$,
\begin{eqnarray}
&&f(T_v^{q:q},\varepsilon,k)=\max\Big\{ g(\chi_{v_{h_q}},\varepsilon,k_1)+F(T_{v_{h_q}},k_2)\Big\},\label{eq-f5}\\
&&s.t.~~~~k_1+k_2=k, \nn\\
&&~~~~~~~~k=1,2,\cdots, \min\big\{|V(T_v^{q:q})|-|Y\cap V(T_v^{q:q})|,K\big\},\nn\\
&&~~~~~~~~\varepsilon\leq k_1=0,1,2,\cdots, \min\big\{|V(\chi_{v_{h_q}})|-1,K\big\},\nn\\
&&~~~~~~~~k_2=0,1, 2, \cdots, \min\big\{ |V(T_{v_{h_q}}|-|Y\cap V(T_{v_{h_q}})|,K\big\},\nn\\
&&F(T_v^{q:q},k)=\max\limits_{\varepsilon=0,1}f(T_v^{q:q},\varepsilon,k).\label{eq-F1}
\end{eqnarray}

If the optimal value $F(T_v^{q:q},k)=f(T_v^{q:q},\varepsilon^*,k)= g(\chi_{v_{h_q}},\varepsilon^*,k_1^*)+F(T_{v_{h_q}},k-k_1^*-\varepsilon^*)$ is obtained when $\varepsilon=\varepsilon^*,k_1=k_1^*$, then the set of upgraded nodes is 
\begin{eqnarray}&&V_v^{q:q}(k)=V_v^{q:q}(\varepsilon^*,k)=V(\chi_{v_{h_q}},\varepsilon^*,k_1^*)\cup V_{v_{h_q}}(k-k_1^*-\varepsilon^*).\label{eq-vf5}\end{eqnarray}
Specially,when $v_{h_q}$ is a leaf node, $E(T_{v_{h_q}})=\emptyset$ and $|E(T_{v_{h_q}})|=0$ and then  $k_2=0$, $F(T_{v_{h_q}},k_2)=F(T_{v_{h_q}},0)=0$. 
\begin{eqnarray}
&&f(T_v^{q:q},\varepsilon,k)=g(\chi_{v_{h_q}},\varepsilon,k),\label{eq-f6}\\
&&F(T_v^{q:q},k)=\max\limits_{\varepsilon=0,1}f(T_v^{q:q},\varepsilon,k).\label{eq-F2}
\end{eqnarray}
If the optimal value $F(T_v^{q:q},k)=f(T_v^{q:q},\varepsilon^*,k)= g(\chi_{v_{h_q}},\varepsilon^*,k)$ is obtained when $\varepsilon=\varepsilon^*$, then the set of upgraded nodes is 
\begin{eqnarray}V_v^{q:q}(k)=V_v^{q:q}(\varepsilon^*,k)=V(\chi_{v_{h_q}},\varepsilon^*,k).\label{eq-v6}
\end{eqnarray}

\textbf{(2.2) When $a=1$, the functions $F(T_v^{1:q},k)$ and $f(T_v^{1:q},\varepsilon,k)$ defined on the non-chain $T_v^{1:q}$}

For all $v\in{V^*}\cup\{v_1\}$, let  
 $F(T_v^{1:q},k)=\max\limits_{\varepsilon=0,1}f(T_v^{1:q},\varepsilon,k)$ be  the optimal value of $T_v^{1:q}$ when $k$ nodes are upgrade with  $\varepsilon=0,1$, $k\geq \varepsilon$  and  $0\leq{k}\leq\min\big\{ |V(T_v^{1:q})|-|Y\cap V(T_v^{1:q})|,K\big\}$, 
 where $f(T_v^{1:q},1,k)$ and $f(T_v^{1:q},0,k)$ represent the cases when  the node $v$ is and is not upgraded, respectively. Let $V_v^{1:q}(k)$, $V_v^{1:q}(0,k)$ and $V_v^{1:q}(1,k)$  be the corresponding sets of edges which are upgraded on $T_v^{1:q}$, respectively.
 Specially, denote $V_v^{1:p}(k)$ by $V_v(k)$ for simplicity.

\textbf{Case (1-1): } when  $k=0$,  then $\varepsilon=0$. 
\begin{eqnarray}
&&F(T_v^{1:q},0)=f(T_v^{1:q},0,0)=\min\Big\{ f(T_v^{q:q},0,0),f(T_v^{1:(q-1)},0,0)\Big\},\label{eq-f6-1}\\
&&V_v^{1:q}(0)=V_v^{1:q}(0,0)=\emptyset.\label{eq-v6-1}
\end{eqnarray}

\textbf{Case (1-2-1): }When $K\geq{k}={|V(T_v^{1:q})|-|Y\cap V(T_v^{1:q})|}$, then $\varepsilon=1$.
\begin{eqnarray}
&&F(T_v^{1:q},k)=f(T_v^{1:q},1,k)=\max\min\Big\{ f(T_v^{q:q},1,k_1),f(T_v^{1:(q-1)},1,k_2)\Big\},\label{eq-f7}\\
&s.t.&k_1+k_2-1=k, \nn\\
&&k_1=1,2,\cdots, \min\big\{|V(T_v^{q:q})|-|Y\cap V(T_v^{q:q})|,K\big\},\nn\\
&&k_2=1,2,\cdots, \min\big\{|V(T_v^{1:(q-1)})|-|Y\cap V(T_v^{1:(q-1)})|,K\big\},\nn\\
&&V_v^{1:q}(k)=V_v^{1:q}(1,k).\label{eq-v7}\end{eqnarray}
\textbf{Case (1-2-2): }When $1\leq{k}\leq K<|V(T_v^{1:q})|-|Y\cap V(T_v^{1:q})|,$
\begin{eqnarray}
&&f(T_v^{1:q},\varepsilon,k)=\max\min\Big\{ f(T_v^{q:q},\varepsilon,k_1),f(T_v^{1:(q-1)},\varepsilon,k_2)\Big\},\label{eq-f8}\\
&s.t.& k_1+k_2-\varepsilon=k, \nn\\
&&k=1,2,\cdots, \min\big\{|V(T_v^{1:q})|-|Y\cap V(T_v^{1:q})|,K\big\},\nn\\
&&\varepsilon\leq k_1=0,1,2,\cdots, \min\big\{|V(T_v^{q:q})|-|Y\cap V(T_v^{q:q})|,K\big\},\nn\\
&&\varepsilon\leq k_2=0,1,2,\cdots, \min\big\{|V(T_v^{1:(q-1)})|-|Y\cap V(T_v^{1:(q-1)})|,K\big\},\nn\\
&&F(T_v^{1:q},k)=\max\limits_{\varepsilon=0,1}f(T_v^{1:q},\varepsilon,k).\label{eq-F3}
\end{eqnarray}
Notice that the special case when $q=1$ with $T_v^{1:0}=\emptyset$ and $f(T_v^{1:(q-1)},\varepsilon,k_2)=0$ has been considered. Under these conditions,  $f(T_v^{1:1},\varepsilon,k)$  can be calculated by formulas (\ref{eq-f1}),(\ref{eq-f2}),(\ref{eq-f3}),(\ref{eq-f4}),(\ref{eq-f5}).

If the optimal value$$\begin{aligned} &F(T_v^{1:q},k)=f(T_v^{1:q},\varepsilon^*,k)=\max\min\Big\{ f(T_v^{q:q},\varepsilon^*,k_1^*),f(T_v^{1:(q-1)},\varepsilon^*,k-k_1^*)\Big\}\end{aligned}$$
 is obtained when $\varepsilon=\varepsilon^*,k_1=k_1^*$, then its set of upgraded nodes is
\eq V_v^{1:q}(k)=V_v^{q:q}(\varepsilon^*,k_1^*)\cup V_v^{1:(q-1)}(\varepsilon^*,k-k_1^*).\label{eq-v8}\nq




To sum up, travel the rooted tree from leaves to the root $v_1$ to calculate all the function values $g(\chi_{v_{h}},\varepsilon,k)$, $F(T_v^{1:q},k)$ and $f(T_v^{1:q},\varepsilon,k)$. Then $F(T_{v_1}, K)$ is the optimal value of the problem \textbf{(MSPIT-UN$_{u}$)}, $S:=V_{v_1}(K)$ is the set of upgraded nodes and an optimal solution is:
\eq \bar{w}(e)=\left\{
\begin{array}{ll}
u(e), &v\in S,e\in A(v),\\
w(e), &otherwise.\end{array}\right.\label{eq-optim-w}\nq

According to the analysis above, we have the following dynamic programming algorithm of the problem \textbf{(MSPIT-UN$_{u}$)}.
\begin{algorithm}
\caption{A dynamic programming algorithm:\\
$ [\bar{w},V_{v_1}(K),F(T_{v_1}, K)]:= $
\textbf{MSPIT-UN$_{u}$}$(T,V,E,Y,w,u,K).$}
\label{alg4}
\begin{algorithmic}[1]

\REQUIRE A tree $T(V,E)$ rooted at $v_1$, the set $Y$ of leaf nodes, two edge weight vectors $w$ and $u$ and the number $K$ of upgrade nodes.

\ENSURE An optimal solution $\bar{w}$ with the set  $V_{v_1}(K)$  of upgraded nodes  and the relative optimal value $F(T_{v_1}, K)$.

\STATE \textbf{(Breath-First Search (BFS).)} Let $V^*=\{v\in{V}|degree(v)>2\}$. Start from the root $v_1$ and label
each node $v$ by $Layer(v)$ when using the breath-first search strategy on $T$, where $Layer(v_1) = 1$. While
executing the BFS, calculate $Layer(v)$ for each node $v$ and $LN(e)$ for each edge $e$. Find the sets $CD(v)$
of critical children for each ${v}\in{V^*\cup\{v_1\}}$.

\STATE for each $v\in V^*\backslash{v_1}\cup Y$,  rearrange the nodes $v_i\in \chi_v$  whose degree are 2 on $\chi_v$ with the same layer number in the same path such that their values of $\Delta w(A(v_i))$ are in descending order.

\FOR {any ${v}\in{V^*\cup\{v_1\}}$}
\FOR {$v_h\in{CD(v)}$}
\STATE Calculate the value  $g(\chi_{v_h},\varepsilon,k)$ and the set $V(\chi_{v_h},\varepsilon,k)$ of upgraded nodes by (\ref{eq-g1v1}), (\ref{eq-g2}), (\ref{eq-v2}), (\ref{eq-g3}) and (\ref{eq-v3}) , for each $k=0,1,\cdots,\min\Big\{K,|V(\chi_{v_h})|-1\Big\}$.
\ENDFOR
\ENDFOR
\FOR { any ${v}\in{V^*\cup\{v_1\}}$ in descending order of the labels $Layer(v)$ of nodes}
\FOR {$q=1:p$}
\FOR {$k=0:\min\Big\{K,|V(T_v^{1:q})|-|Y\cap V(T_v^{1:q})|\Big\}$, }
 \STATE Calculate the value $f(T_v^{q:q},\varepsilon,k)$ and $F(T_v^{q:q},k)$ and their sets of upgrade nodes by (\ref{eq-f1})-(\ref{eq-v6}).
\STATE  Calculate the value $f(T_v^{1:q},\varepsilon,k)$ and $F(T_v^{1:q},k)$ and their sets of upgrade nodes by (\ref{eq-f6-1})-(\ref{eq-v8}).
\ENDFOR
\ENDFOR
\ENDFOR

\STATE  For the root $v_1$, calculate the optimal value $F(T_{v_1}, K)$, the relative set of upgraded nodes $V_{v_1}(K)$ and an optimal solution $\bar w$ obtained from (\ref{eq-optim-w}).
\end{algorithmic}
\end{algorithm}



\begin{thm}
The problem \textbf{(MSPIT-UN$_u$)} can be solved in $O(n^3)$ time by a dynamic programming algorithm in \textbf{Algorithm \ref{alg4}}.
\end{thm}
\begin{proof}
In Algorithm \ref{alg4}, labeling the tree $T$ by breadth first search and determining the critical descendants for any $v\in V^*\cup\{v_1\}$ in \textbf{Line 1}  can be completed in $O(n)$ time.

The calculation of the value of $\Delta w(e)$ in Line 2 can be completed in $O(n)$ time.  Rearranging the edges with the same layer number in the same path such that their values of $\Delta w(e)$ are in descending order can be finished in time  $O(n\log n)$. Thus, the total time of \textbf{Line 2 } is $O(n\log n)$.



In \textbf{Lines 3-7}, on the one hand, for a given chain $\chi_{v_{h}}$, $g(\chi_{v_{h}},\varepsilon,k)$ can be obtained  from fomulas (\ref{eq-g1v1}),(\ref{eq-g2}) and (\ref{eq-g3}) in $O(|\chi_{v_{h}}|)$  at most. On the other hand, in fomula (\ref{eq-g3}) $$g(\chi_{v_h},0,k+1)= g(\chi_{v_h},0, k)+\Delta w(e_{i_{k+2}})  \text{~and~ } g(\chi_{v_h},1,k+1)= g(\chi_{v_h},1, k)+\Delta w(e_{i_{k+1}})$$ holds. Thus, $g(\chi_{v_{h}},\varepsilon,k)$ for any $k=0, 1, \cdots, \min\Big\{K,|V(\chi_{v_h})|-1\Big\}$ all can be calculated in $O(|\chi_{v_{h}}|)$. Additionally, for all $v\in V^*\cup\{v_1\}$ and  the relevant $v_h\in CD(v)$, calculating every $g(\chi_{v_{h}},\varepsilon,k)$ is just travelling the edges in every chain. Hence, the total time of \textbf{Lines 3-7} is $\sum_{\chi\subseteq T} O(|\chi|)=O(n).$


In \textbf{Lines 8-15},  the functions $F(T_v^{q:q},k)$ and $f(T_v^{q:q},\varepsilon,k)$ defined on the non-chain $T_v^{q:q}$ can be obtained from  formulas (\ref{eq-f1}), (\ref{eq-f2}), (\ref{eq-f3}), (\ref{eq-f4}), (\ref{eq-f5}), (\ref{eq-F1}), (\ref{eq-f6}), (\ref{eq-F2}).  In the most complex case in (\ref{eq-f5}), for a given non-chain $T_v^{q:q}$ and a value $k$,
since $k_1+k_2=k$, then $k_1, k_2\leq k$ and $k_2$ is obtained when $k_1$ is given which contains $O(k)$ kinds of possible combinations and for each pair of combination, there is only one addictive operation. Thus, for a given non-chain $T_v^{q:q}$ and a value $k$, $F(T_v^{q:q},k)$ and $f(T_v^{q:q},\varepsilon,k)$  can be solved in $O(k)$. Then for all $k=0, 1, \cdots, \min\Big\{K,|V(T_v^{1:q})|-|Y\cap V(T_v^{1:q})|\Big\}$, $F(T_v^{q:q},k)$ and $f(T_v^{q:q},\varepsilon,k)$ can be completed in $O(K^2)$ time. There are at most $O(n)$ non-chains, hence, all these functions can be obtained in $O(nK^2)$ time.

The functions $f(T_v^{1:q},\varepsilon,k)$ and $F(T_v^{1:q},k)$  defined on the non-chain $T_v^{1:q}$ can be obtained from  formulas (\ref{eq-f6-1}), (\ref{eq-f7}),  (\ref{eq-f8}),(\ref{eq-F3}).In the most complex case in (\ref{eq-f8}),  for a given non-chain $T_v^{1:q}$ and a value $k$,
since $k_1+k_2-\varepsilon=k$, then $k_1, k_2\leq k$ and $k_2$ is obtained when $k_1$ is given which contains $O(k)$ kinds of possible combinations and for each pair of combination, there is only one comparison. For all $k=0,1,2,\cdots, \min\big\{|V(T_v^{1:q})|-|Y\cap V(T_v^{1:q})|,K\big\}$, $f(T_v^{1:q},\varepsilon,k)$ and $F(T_v^{1:q},k)$ can be calculated in $O(K^2)$ time.
For any $v\in V^*\cup\{v_1\}$, $q=1,2 \cdots, deg(v)-1$, there are  $\sum_{v\in V^*\cup\{v_1\}}(deg(v)-1)\leq 2m-|V^*| \leq 2m $ subproblems like $f(T_v^{1:q},\varepsilon, k)$. Therefore, the total time complexity of \textbf{Lines 8-15} is $O(nK^2)$.

As a conclusion, the time complexity of Algorithm \ref{alg4} is $
 O(n K^2) \leq O(n^3).
$
\end{proof}

\section{Slove the problem \textbf{(MCSPIT-UN$_u$)}}
Now we consider a minimum cost shortest path interdiction problem by upgrading
nodes on trees under unit cost, which is similarly denoted by
\textbf{(MCSPIT-UN$_u$)}. It aims to minimize the total number of upgrade nodes on the premise
that the shortest root-leaf distance of the tree is lower bounded by a given value $D$.
\begin{eqnarray}
&\min\limits_{S\subseteq V}&|S|\nonumber\\
\textbf{(MCSPIT-UN$_u$)}&s.t.& \min_{t_k\in{Y}}\tilde {w}(P_k)\geq D,\label{eq-MCSPIT-UN}\\
&&\tilde w(e)=\left\{\begin{array}{cc}
  u(e),   & v\in S,e\in A(v) \\
   w(e),  & otherwise
\end{array}\right..\nonumber
\end{eqnarray}
For convenience, denote by \textbf{(MSPIT-UN$_u$(K))} and \textbf{(MCSPIT-UN$_u$(D))} the problem
\textbf{(MSPIT-UN$_u$)} with a given $K$ and the problem \textbf{(MCSPIT-UN$_u$)} with a given $D$, respectively.
The problem \textbf{(MSPIT-UN$_u$(K))} can be solved by \textbf{Algorithm 1} for a given $K$ in Sect.
2. In the problem \textbf{(MCSPIT-UN$_u$(D))}, we are searching for the smallest $K^*$ such that the problem \textbf{(MSPIT-UN$_u$($K^*$))} generates an upgrade vector $w^*$  with $\min_{t_k\in Y}w^*(P_k)\geq D$. Furthermore, we can obviously observe that for any $D'$ and $D''$ with $D'<D''$, the
number of upgrade nodes for the problem \textbf{(MCPIT-UN$_u$($D'$
))} is no more than that for \textbf{(MCPIT-UN$_u$($D''$))}.

To solve the problem \textbf{(MCSPIT-UN$_u$)}, we aim to find the optimal $K^*$ among the
values $\{1,2,\cdots,n\}$ by a binary search method, and in each iteration we solve
a problem \textbf{(MSPIT-UN$_u$(k))} by \textbf{Algorithm 1}, in which $k$ is the median of the current
interval $ [k_1, k_2] \subseteq [1, n]$. Hence, the problem \textbf{(MCSPIT-UN$_u$)} can be solved in
$O(n^3 \log n)$, as shown in Algorithm
2.

\begin{algorithm}
\caption{\\
$ [\tilde w,S,K^*]:=$ \textbf{MCSPIT-UN$_{u}$}$(T,V,E,Y,w,u,D).$}
\label{MCSPITUH}
  \begin{algorithmic}[1]
 \REQUIRE A tree $T(V,E)$ rooted at $v_1$, the set $Y$ of leaf nodes, two edge weight vectors $w$ and $u$ and a given value  $D$.

\ENSURE An optimal solution $\tilde {w}$ with the set $S$ of upgraded nodes    and the relative optimal value $K^*$.
  \STATE Initialization: $k_1:=1, k_2:=n,i:=1$
  \WHILE{$k_2\neq k_1+1$}
  \STATE Let $k:=\lceil\frac{k_1+k_2}{2}\rceil$
    \STATE Call $[\bar{w},\bar V,D^i]:= \textbf{MSPIT-UN$_{u}$}(T,V,E,Y,w,u,k)$
    \IF{$D^i<D$}
    \STATE Let $k_1:=k.$
    \ELSIF{$D^i>D$}
    \STATE Let $k_2:=k.$
    \ELSE
    \RETURN ($\bar w,\bar V,k$)
    \ENDIF
    \STATE
    Update $i:=i+1$
    \ENDWHILE
     \STATE  Call $[\bar{w},\bar V,D^i]:=$ \textbf{MSPIT-UN$_{u}$} $(T,V,E,Y,w,u,k_2)$
     
     \RETURN ($\bar w,\bar V,k_2$)
  \end{algorithmic}
\end{algorithm}
\section{Computational experiments}
\subsection{An example to show Algorithm \ref{alg4}.}
For a better understanding of Algorithm \ref{alg4}, \textbf{Example 1} is given to show the detailed computing process.

\textbf{Example 1} As is shown in Fig. \ref{LNCCV}, let $V=\{v_1,v_2,\cdots,v_{10}\}$, $E=\{e_2,\cdots,e_{10}\}$, $t_1=v_3$, $t_2=v_4$, $t_3=v_6$,$t_4=v_8$,$t_5=v_{10}$, $w=(6,6,4,8,1,4,3,4,5)$, $u(e_j)=10$ for all $j=1,2,\cdots,10$ and $K=1$.

\textbf{In Lines 1-2:}  $V^*\cup\{v_1\} := \{v_1, v_2, v_7\}$, and the sets of critical descendant are $CD(v_1) :=\{v_2, v_6, v_7\}, CD(v_2) = \{v_3, v_4\}, CD(v_7) := \{v_8, v_{10}\}$.

\textbf{In Lines 3-7:} For every $V^*\cup\{v_1\}$ and $v_h \in CD(v)$, calculate the value  $g(\chi_{v_h},\varepsilon,k)$ and the set $V(\chi_{v_h},k)$ of upgraded nodes by (\ref{eq-g1v1}), (\ref{eq-g2}), (\ref{eq-v2}), (\ref{eq-g3}) and (\ref{eq-v3}) , for each $k=0,1,\cdots,\min\Big\{K,|V(\chi_{v_h})|-1\Big\}$. The  values of $g(\chi_{v_h},\varepsilon,k)$
and the set $V(\chi_{v_h},k)$ of upgrade nodes are shown in \textbf{Table \ref{table-g}}.


 


 
 
 
 
 
 
 
 
 
 
 
 

\begin{table}[htbp!]
    \centering
    \caption{The value  $g(\chi_{v_h},\varepsilon,k)$ and the corresponding set $V(\chi_{v_h},\varepsilon,k)$. }
    \label{table-g} 
    \begin{tabular}{ccccc|ccccc}\hline 
       $\chi$&$ \varepsilon$&$k$&$g$$$&$V(g)$& $\chi$&$ \varepsilon$&$k$&$g$$$&$V(g)$\\
       \hline
$\chi_{v_3}$& $0$&$0$ & $7$&$\emptyset$&$\chi_{v_8}$& $0$&$0$ & $3$&$\emptyset$\\  \cline{2-5}\cline{7-10}
&$1$&$1$ & $10$&$\{v_2\}$&& $1$&$1$ & $10$&$\{v_7\}$\\  \hline
$\chi_{v_4}$& $0$&$0$ & $4$&$\emptyset$&$\chi_{v_{10}}$& $0$&$0$ & $9$&$\emptyset$\\  \cline{2-5}\cline{7-10}
&$1$&$1$ & $10$&$\{v_2\}$&& $0$&$1$ & $14$&$\{v_9\}$\\\cline{1-5}\cline{7-10}
$\chi_{v_2}$& $0$&$0$ & $6$&$\emptyset$&& $1$&$1$ & $15$&$\{v_7\}$\\  \cline{2-5}\cline{7-10}
&$1$&$1$ & $10$&$\{v_1\}$&& $1$&$2$ & $20$&$\{v_7,v_9\}$\\  \hline
$\chi_{v_6}$& $0$&$0$ & $9$&$\emptyset$&$\chi_{v_7}$& $0$&$0$ & $4$&$\emptyset$\\  \cline{2-5}\cline{7-10}
& $0$&$1$ & $11$&$\{v_5\}$&& $1$&$1$ & $10$&$\{v_1\}$\\  \cline{2-10}
&$1$&$1$ & $18$&$\{v_1\}$\\  \cline{2-5}
&$1$&$2$ & $20$&$\{v_1,v_5\}$\\  \hline
    \end{tabular}
\end{table}

\textbf{In Lines 8-15:} for { any ${v}\in{V^*\cup\{v_1\}}$ in descending order of the labels $Layer(v)$ of nodes}
and  {$k=0:\min\Big\{K,|V(T_v^{1:q})|-|Y\cap V(T_v^{1:q})|\Big\}$, }
calculate the value $f(T_v^{q:q},\varepsilon,k)$ and $F(T_v^{q:q},k)$ and their sets of upgrade nodes by 
 (\ref{eq-f1}), (\ref{eq-vf1}), (\ref{eq-f2}), (\ref{eq-vf2}), (\ref{eq-f3}), (\ref{eq-vf3}), (\ref{eq-f4}), (\ref{eq-vf4}), (\ref{eq-f5}), (\ref{eq-F1}), (\ref{eq-vf5}), (\ref{eq-f6}),(\ref{eq-F2}), (\ref{eq-v6}). 
Calculate the value $f(T_v^{1:q},\varepsilon,k)$ and $F(T_v^{1:q},k)$ and their sets of upgrade nodes by (\ref{eq-f6-1}),(\ref{eq-v6-1}), (\ref{eq-f7}), (\ref{eq-v7}), (\ref{eq-f8}),(\ref{eq-F3}), (\ref{eq-v8}). 

$V^*\cup\{v_1\} := \{v_1, v_2, v_7\}$.

  The values  $f(T_v^{a:q},\varepsilon,k)$  with the corresponding sets of upgrade nodes  are shown in \textbf{Table \ref{table-f}}, where the the bolded content indicates the corresponding function values of $F(T_v^{1:q},k)$ and their associated sets of upgrade nodes.


\begin{table}[htbp!]
    \centering
    \caption{The values  $f(T_v^{a:q},\varepsilon,k)$  with the corresponding sets of upgrade nodes. }
    \label{table-f} 
    \begin{tabular}{ccccc|ccccc}\hline 
       $T_{v}^{a:q}$&$ \varepsilon$&$k$&$f$&$V(T)$& $T_{v}^{a:q}$&$ \varepsilon$&$k$&$f$$$&$V(T)$\\
       \hline
$T_{v_2}^{1:1}$& $0$&$0$ & $\bm 4$&$\bm\emptyset$&$T_{v_1}^{1:1}$&$0$&$0$&$\bm7$&$\bm\emptyset$\\  \cline{2-5}\cline{7-10}

&$1$&$1$ & $\bm{10}$&$\bm{\{v_2\}}$&&$0$&$1$&$\bm{14}$&$\bm{\{v_7\}}$\\  \cline{1-5}\cline{7-10}
$T_{v_2}^{2:2}$& $0$&$0$ & $ 7$&$\emptyset$&&$1$&$1$&${13}$&${\{v_1\}}$\\  \cline{2-5}\cline{6-10}
&$1$&$1$ & ${10}$&${\{v_2\}}$&$T_{v_1}^{2:2}$&$0$&$0$&$9$&$\emptyset$\\  \cline{1-5}\cline{7-10}
$T_{v_7}^{1:1}$& $0$&$0$ & $\bm9$&$\bm\emptyset$&&$0$&$1$&$11$&$\{v_5\}$\\  \cline{2-5}\cline{7-10}
& $0$&$1$ & $14$&$\{v_9\}$&&$1$&$1$&${18}$&${\{v_1\}}$\\  \cline{2-5}\cline{6-10}
&$1$&$1$ & $\bm{15}$&$\bm{\{v_7\}}$&$T_{v_1}^{3:3}$&$0$&$0$&${10}$&$\emptyset$\\  \cline{1-5}\cline{7-10}
$T_{v_7}^{2:2}$& $0$&$0$ & $3$&$\emptyset$&&$0$&$1$&${16}$&${\{v_2\}}$\\  \cline{2-5}\cline{7-10}
&$1$&$1$ & ${10}$&${\{v_7\}}$&&$1$&$1$&$14$&$\{v_1\}$\\  \hline
$T_{v_2}^{1:2}$& $0$&$0$ & $\bm4$&$\bm\emptyset$&$T_{v_1}^{1:2}$&$0$&$0$&$7$&$\emptyset$\\  \cline{2-5}\cline{7-10}
& $1$&$1$ & $\bm{10}$&$\bm{\{v_2\}}$&&$0$&$1$&$7$&$\{v_7\}$\\  \cline{1-5}\cline{7-10}
$T_{v_7}^{1:2}$& $0$&$0$ & $\bm3$&$\bm\emptyset$&&$1$&$1$&$13$&$\{v_1\}$\\  \cline{2-5}\cline{6-10}
& $0$&$1$ & $3$&$\emptyset$&$T_{v_1}^{1:3}$&$0$&$1$&$9$&$\{v_7\}$\\  \cline{2-5}\cline{7-10}
 &$1$&$1$ & $\bm{10}$&$\bm{\{v_7\}}$&&$1$&$1$&$13$&$\{v_1\}$\\  \hline
    \end{tabular}
\end{table}

Consequently, \text{From ~(\ref{eq-F3}),(\ref{eq-v8}),}\begin{eqnarray*}&&F(T_{v_1},1)=\max\limits_{\varepsilon=0,1}f(T_{v_1},\varepsilon,1)=\max\{9,13\}=13, S:=V_{v_1}(1)=\{v_1\}.\end{eqnarray*}
   From (\ref{eq-optim-w}), an optimal solution is:
\begin{eqnarray*}
  \bar{w}(e)=\left\{
\begin{array}{ll}
u(e), &v\in S,e\in A(v),\\
w(e), &otherwise.\end{array}\right..  
\end{eqnarray*} 

\subsection{Numerical experiments} 

To evaluate the performance of Algorithms~\ref{alg4} and~\ref{MCSPITUH}, we conducted numerical experiments using MATLAB 2025a on a Windows 11 system with an Intel Core i7-10875H CPU (2.30 GHz). Six randomly generated tree instances, varying in size from 100 to 3,000 vertices, were used as test cases. For each tree, input data ($u$,  $w$) were generated randomly, respecting the constraints $0 \leq w \leq u$. The parameters $K$ and $D$ were randomly chosen, scaling with the tree size $n$. The numerical performance results are presented in Table~\ref{table:ex}.

The table reports the average ($ T_i $), maximum ($ T_i^{\text{max}} $), and minimum ($ T_i^{\text{min}} $) CPU execution times for $ i = 1, 2 $, corresponding to Algorithms~\ref{alg4} and~\ref{MCSPITUH}, respectively.  
The results indicate that both algorithms exhibit efficient performance on large-scale trees, consistent with their theoretical time complexity. Notably, Algorithm~\ref{MCSPITUH} incurs longer CPU times than Algorithm~\ref{alg4}, as expected from its iterative structure: the while-loop in Algorithm~\ref{MCSPITUH} necessitates repeated calls to Algorithm~\ref{alg4}, thereby accumulating computational overhead.

\begin{table}[htbp!]
    \centering
    \caption{Performance of Algorithms \ref{alg4} and \ref{MCSPITUH}.}
    \label{table:ex} 
    \begin{tabular}{ccccccc}\toprule   
        \textbf{Complexity} & \textbf{$ n $} & \textbf{100} & \textbf{500} & \textbf{1000} & \textbf{2000} & \textbf{3000} \\ 
        \midrule
        $ O(n^3) $ & $ T_1 $ & 0.0017 & 0.2178 & 1.7427 & 13.9416 & 47.3142 \\
                         & $ T_1^{\text{max}} $ & 0.0031 & 0.4392 & 3.5134 & 25.1072 & 89.2314 \\
                         & $ T_1^{\text{min}} $ & 0.0003 & 0.0423 & 0.4184 & 2.5472 & 9.5106 \\
        \midrule    
        $ O(n^3 \log n) $ & $ T_2 $ & 0.0047 & 0.8426 & 7.0352 & 62.3792 & 220.5431 \\
                         & $ T_2^{\text{max}} $ & 0.0062 & 1.4413 & 9.8624 & 81.3716 & 292.9127 \\
                         & $ T_2^{\text{min}} $ & 0.0022 & 0.3586 & 3.2691 & 28.9880 & 130.1103 \\        
        \bottomrule
    \end{tabular}
\end{table}

\section{Conclusion and further research}

In this paper, we investigate the maximum shortest path interdiction problem by upgrading nodes on tree network under unit cost (MSPIT-UN$_u$). The objective is to upgrade a subset of nodes to maximize the length of the shortest root-leaf distance, given that the total upgrade cost is bounded by a predetermined value. We develop a dynamic programming algorithm with time complexity $O(n^3)$ to solve this problem efficiently.

Additionally, we address the related Minimum Cost variant (MCPIT-UN$_u$) and propose a binary search algorithm with time complexity $O(n^3\log n)$, where our dynamic programming algorithm is executed in each iteration to solve the corresponding MSPIT-UN$_u$ problem.

For future research, several promising directions can be explored. First, the MSPIT-UN problem can be generalized by considering variable cost vectors for node upgrades, rather than restricting to unit costs. Second, the problem could be extended to more complex network structures, such as series-parallel graphs or general graphs, to broaden its applicability. Finally, other network interdiction problems involving the upgrading of critical nodes, such as minimum spanning tree interdiction problems, present interesting avenues for investigation. These extensions would enhance both the theoretical understanding of the problem and expand its practical relevance in real-world scenarios.

\vskip 0.3cm

\section*{Declarations}

\textbf{Competing interests} The authors declare that they have no competing interest.

\end{document}